\documentclass[12pt]{article}
\usepackage[english]{babel}

\usepackage{amsmath}
\usepackage{amssymb,amsthm}

\usepackage{mathrsfs}

\usepackage{bbm}

\numberwithin{equation}{section}

\newtheorem{theorem}{Theorem}

\newtheorem{lemma}{Lemma}[section]

\newtheorem{remark}[lemma]{Remark}

\newtheorem{notation}[lemma]{Notation}

\renewcommand{\leq}{\leqslant}
\renewcommand{\geq}{\geqslant}

\newcommand{\Rot}{\mathcal{R}}

\newcommand{\ChEps}{\varrho}

\makeatletter
\newcommand*{\IfItalic}{%
  \ifx\f@shape\my@test@it
    \expandafter\@firstoftwo
  \else
    \expandafter\@secondoftwo
  \fi
}
\newcommand*{\my@test@it}{it}
\makeatother

\newcommand{\myae}{\IfItalic{\emph{\mbox{\ae}}}{\mbox{\ae}}}

\usepackage[usenames]{color}
\usepackage{colortbl}
\usepackage{ifthen}

\newcommand\lI{L}

\newcommand\ssemg\psi
\newcommand\fkinkg\gamma
\newcommand\iks[1]{\widehat{x}_{#1}}
\newcommand\iksp[1]{\widehat{x}_{#1}^{+}}
\newcommand\ikspm[1]{\widehat{x}_{#1}^{\, \pm}}
\newcommand\iksm[1]{\widehat{x}_{#1}^{\, -}}
\newcommand\iksmp[1]{\widehat{x}_{#1}^{\, \mp}}
\newcommand\dGt{
\beta}

\usepackage[all]{xy}

\usepackage{amscd}

\begin{document}

\begin{center}
{\Large Infinite sequence on 2-element alphabet\\
as coordinates of points, determines by unimodal maps}

{\large Makar Plakhotnyk\\
University of S\~ao Paulo, Brazil.\\

makar.plakhotnyk@gmail.com}\\
\end{center}

\begin{abstract}
We study in this work different $(0, 1)$-codings of points from
the unit interval $[0, 1]$ in the relation with the treatment of
continuous unimodal maps.~\footnote{This work is partially
supported by FAPESP (S\~ao Paulo, Brazil).}~\footnote{AMS subject
classification:
37E05  
}~\footnote{Key words: one-dimensional dynamics, invariant
coordinates}
\end{abstract}

\section{Introduction}

The final aim of the theory of the motions of dynamical systems
must be directed to toward the qualitative determination of all
possible types of motions and of the interrelation of these
motions~\cite[p. 189]{Birk-1927}. Topological conjugation is the
classical tool to divide the dynamical systems to the classes of
equivalence, where all the trajectories are the same in a certain
cense.

We will call a continuous map $g: [0, 1]\rightarrow [0, 1]$
\textbf{unimodal}, if it can be written in the form
\begin{equation}\label{eq:1.1} g(x) = \left\{
\begin{array}{ll}g_l(x),& 0\leq x\leq
v,\\
g_r(x), & v\leq x\leq 1,
\end{array}\right.
\end{equation} where %
$v\in (0,\, 1)$ is a parameter, the function $g_l$ is increasing,
the function $g_r$ is decreasing, and $$g(0)=g(1)=1-g(v)=0.$$

If a homeomorphism $h$ satisfies the functional equation
$$h\circ g_1 = g_2\circ h,
$$ %
where $g_1$ and $g_2$ are unimodal maps, then we will say that $h$
is a conjugation from $g_1$ to $g_2$.

For a fixed unimodal map $g:\, [0, 1]\rightarrow [0, 1]$ and for
each point $x\in [0, 1]$ we will construct different $(0, 1)$
sequences, and each of them will be considered as coordinates of
$x$ in some cense. These sequences will be invariant with respect
to topological conjugateness of unimodal maps, and we will use
them to construct the explicit formula of the topological
conjugacy. The usefulness of such approach is well known in the
one-dimensional dynamics. It comes from~\cite{Milnor-Thurston} by
J. Milnor and W.~Thurston (1988), and is widely used in the
Kneading theory, created by them.

Let $g$ be unimodal map. For any $t\in [0, 1]$ write
$\varepsilon(t)=1$ if $t<v$, $\varepsilon(v)=0$, and
$\varepsilon(t)=-1$ if $t>v$. Now for any $x\in [0, 1]$ denote $$
\theta_n = \cdot\prod\limits_{i=0}^{n}\varepsilon(g^n(x)).
$$ Following~\cite{Milnor-Thurston}, the sequence $$
x \longleftrightarrow \theta(x) =\theta_0\theta_1\ldots
\theta_n\ldots
$$ is %
called the \textbf{invariant coordinates} of $x$ (with respect to
$g$).

\begin{lemma}\cite[p. 478, Lemma~3.1]{Milnor-Thurston}
The map $x\mapsto \theta(x)$ is increasing, where the natural
lexicographical order is defined on the set of sequences
$\{(\theta_i)_{i\geq 0}\}$ of elements $\{-1; 0; 1\}$.
\end{lemma}

The classical example of the unimodal maps is the \textbf{tent}
map $$ f:\, x\mapsto 1-|1-2x|.
$$

\begin{theorem}\label{th:1}\cite[p. 53]{Ulam-1964-b} %
A unimodal map $g$ is topologically conjugated to the tent $f$ map
if and only if the complete pre-image of $0$ under the action of
$g$ is dense in $[0,\, 1]$.
\end{theorem}

Recall, that the set $g^{-\infty}(a) = \bigcup\limits_{n\geq
1}g^{-n}(a)$, where $g^{-n}(a) = \left\{ x\in [0,\, 1]:\, g^n(x) =
a\right\}$ for all $n\geq 1$, is called the complete pre-image of
$a$ (under the action of the map $g$). We have shown in details
in~\cite{Chaos} that the original proof of Theorem~\ref{th:1} uses
the next construction.

\begin{notation}\cite[Notation~2.1]{Chaos}
Let $g$ be unimodal map. Then the set $g^{-n}(0)$ consists of
$2^{n-1}+1$ points. Thus, for every $n\geq 1$ denote $\{
\mu_{n,k}(g),\, 0\leq k\leq 2^{n-1}\}$ such that
$g^n(\mu_{n,k}(g))=0$ and $\mu_{n,k}(g)<\mu_{n,k+1}(g)$ for all
$k$.
\end{notation}

\begin{lemma}\cite[Lemma 3]{Chaos}
Suppose that $g_1,\, g_2:\, [0,\, 1]\rightarrow [0,\, 1]$ are
unimodal maps, and $h$ is the conjugacy from $g_1$ to $g_2$. Then
$$
h(\mu_{n,k}(g_1)) =\mu_{n,k}(g_2) $$ for all $n\geq 1$ and $k,\,
0\leq k\leq 2^{n-1}$.
\end{lemma}

We have studied the properties of the sequence $\mu_{n,k}(g)$ in
our~\cite{Chaos}.

\begin{notation}\cite[Notation~4.4]{Chaos}
For every $n\geq 1$ denote by $k_n$ the maximal $k\in \{ 0,\ldots,
2^{n-1}-1\}$ such that $x\in [\mu_{n+1,k}, \mu_{n+1,k+1}]$. For
every $n\geq 1$ and $k,\, 0\leq k\leq 2^{n-1}-1$ denote $\lI_{n,k}
= \mu_{n,k+1} -\mu_{n,k}$. Write $\iks{n} = \mu_{n+1,k_n}$.
\end{notation}

\begin{lemma}\cite[Lemma~7]{Chaos}
For every $x\in [0, 1]$ there exists a sequence $(\dGt_i)_{i\geq
0}$ of $\dGt_i\in \{0; 1\}$ with $\dGt_0=0$, such that $k_n =
\sum\limits_{i=1}^{n}\dGt_i2^{n-i}$ for all $n\geq 1$.
\end{lemma}

We call the sequence $(\dGt_i)_{i\geq 0}$ the
\textbf{$g$-decomposition} of $x$.

\begin{remark}
If an unimodal map $g$ is the tent map $f$, then the
$f$-decomposition of any $x\in [0, 1]$ equals the classical binary
decomposition of $x$.
\end{remark}

Another example of the use of $(0, 1)$ sequences as coordinates of
points, determined by unimodal maps, appeared
in~\cite{Yong-Guo-Wang} in the study of the properties of the
topological conjugacy of unimodal maps in the case, when graphs of
both $g_l$ and $g_r$ in~\eqref{eq:1.1} are line segments, i.e. $g$
is of the form \begin{equation}\label{eq:1.2} f_v(x) =
\left\{\begin{array}{ll}
\frac{x}{v},& \text{if }0\leq x\leqslant v,\\
 \frac{1-x}{1-v},& \text{if }
v\leq x\leq 1.
\end{array}\right.
\end{equation}
The topological conjugation of maps of the form~\eqref{eq:1.2} are
studied since early 1980-th, see~\cite{Jydy}, \cite{Skufca}
and~\cite{Yong-Guo-Wang}. Following~\cite{Skufca}
and~\cite{Yong-Guo-Wang} we will call~\eqref{eq:1.2} \textbf{skew
tent map}. It was proved in~\cite{Jydy} that for every $v_1,\,
v_2\in (0, 1)$ the maps $f_{v_1}$ and $f_{v_2}$ are topologically
conjugated. For any $x\in [0, 1]$ and unimodal map $g$ denote
$$\ChEps_n(g, x) = \left\{\begin{array}{ll} 0& \text{if }
g^{n-1}(x)\leq
\mu_{2,1}(g),\\
1& \text{otherwise}
\end{array}\right.$$ %
for all $n\geq 1$. Following~\cite{Yong-Guo-Wang}, we will call
the sequence $(\ChEps_n(g, x))_{n\geq 1}$ the \textbf{$g$-digit
sequence of~$x$, with respect to $g$}. Notice, that the elements
of $g$-digit sequence are denoted by $\varepsilon$-s
in~\cite{Yong-Guo-Wang}, but we have changed this notation to
avoid the confusion with $\varepsilon$ in the construction of
Minor-Thurston's invariant coordinates. The next theorem was
stated in~\cite{Yong-Guo-Wang}.

\begin{theorem}\cite[Theorem~2]{Yong-Guo-Wang}\label{th:2}
Let $x\in [0, 1]$, let $v_1,\, v_2\in (0, 1)$ be arbitrary, and
$h$ be the conjugacy from $f_{v_1}$ to $f_{v_2}$. Denote by
$(\ChEps_n)_{n\geq 1}$ the $f_{v_1}$-digit sequence of $x$. Then:

\noindent (1) The equality
$$
x = \ChEps_1 +\sum\limits_{n=2}^\infty \ChEps_n\cdot
v_1^{n-1}\cdot\left(\frac{v_1-1}{v_1}\right)^{\sum\limits_{j=1}^{n-1}
\ChEps_j} $$ holds.

\noindent (2) The $f_{v_2}$-digit sequence of $h(x)$ is
$(\ChEps_n)_{n\geq 1}$.

\noindent (3) The equality
$$
h(x) = \ChEps_1 +\sum\limits_{n=2}^\infty \ChEps_n\cdot
v_2^{n-1}\cdot
\left(\frac{v_2-1}{v_2}\right)^{\sum\limits_{j=1}^{n-1} \ChEps_j}
$$ holds.
\end{theorem}

Notice, that (2) of Theorem~\ref{th:2} follows from the basic
properties of the topological conjugacy, and (3) follows from (1)
and (2). In other words, the unique non-trivial part of
Theorem~\ref{th:2} is~(1). We will generalize Theorem~\ref{th:2}.
We will prove

\begin{theorem}\label{th:3}
Let $g$ be unimodal map and let $(\ChEps_i)_{i\geq 1}$ be the
$g$-digit sequence of a point $x\in [0, 1]$. For any $n\geq 1$
denote
$$a_n = \iks{n} + \dGt_{n+1}\cdot \lI_{n+1,k_n}.$$ Then $$ a_n =
\ChEps_1 +\sum\limits_{i=2}^{n+1}
(-1)^{\sum\limits_{j=1}^{i-1}\ChEps_j}\cdot \ChEps_i \cdot
\lI_{i,k_{i-1}}.
$$
\end{theorem}

\begin{theorem}\label{th:4}
For any $v\in (0, 1)$ and a point $x\in [0, 1]$ with $f_v$-digit
sequence $(\ChEps_i)_{i\geq 1}$ we have that $$\lI_{n+1,k_n}
=(1-v)^{\sum\limits_{i=1}^{n} \varrho_i}\cdot
v^{n-\sum\limits_{i=1}^{n} \varrho_i}.$$
\end{theorem}

Now Theorem~\ref{th:2} follows from Theorems~\ref{th:3}
and~\ref{th:4}.

We prove Theorem~\ref{th:4} at the end of Section~\ref{skew}, and
we prove Theorem~\ref{th:3} at the end of Section~\ref{explicit}.
We state the relation between invariant coordinates,
$g$-decomposition and $g$-digit sequences in Section~\ref{relat}.

\section{Preliminaries}

We have introduced the notations from this section in
our~\cite{Chaos} in generalization of the original proof of
Theorem~\ref{th:1}.

\begin{notation}\cite{Chaos}\label{not:2.1}
1. For every $n\geq 0$ let $k_n,\, 0\leq k_n< 2^{n-1}$ be minimal
such that $x\in [\mu_{n+1,k_n}, \mu_{n+1,k_n+1}]$. Denote $\iks{n}
=\mu_{n+1,k_n}, $ $\iksp{n} = \mu_{n+1,k_n+1},$ $\ikspm{n} =
\mu_{n+2,2k_n+1}, $ $ \iksm{n} =\mu_{n+1,k_n-1}$ and $\iksmp{n}
=\mu_{n+2,2k_n-1}$.

2. For any $n\geq 1$ and $k,\, 0\leq k< 2^{n-1}$ denote $I_{n,k} =
(\mu_{n,k},\, \mu_{n,k+1})$ and $\lI_{n,k} =
\mu_{n,k+1}-\mu_{n,k}$.
\end{notation}

\begin{notation}
For every $i,j:\, 0\leq i<j$ denote $k_{i,j} =
\sum\limits_{t=i}^{j}\dGt_t2^{j-t}$. Remark that $k_n = k_{0,n} =
k_{1,n}$.
\end{notation}

\begin{remark}
Notice, that, by Notation~\ref{not:2.1},\\ $I_{n+1,k_n} =
(\iks{n}, \iksp{n})$ and\\
$I_{n+2,2k_n} = (\iks{n}, \ikspm{n})$.
\end{remark}

\begin{notation}\label{not:2.4}
For any $n\geq 1$ and $k,\, 0\leq k< 2^n$ denote
$$\delta_{n,k} =
\frac{\mu_{n+1,2k+1}-\mu_{n,k}}{\mu_{n,k+1}-\mu_{n,k}}.$$ If $I$
is an interval of the form $I = I_{n,k}$ then write $\delta(I)$
for $\delta_{n,k}$. This will be convenient, for example, for the
expression $\delta(g(I_{n,k}))$.
\end{notation}

\begin{notation}
1. For every $t\in [0, 1]$ denote $\Rot(t) = 1-t$.

2. For every $n\geq 1$ and $t\in \{0,\ldots, 2^{n}-1\}$ denote
$\Rot_n(t) = 2^n-t-1$.
\end{notation}

\begin{remark}\cite[Rem.~2.14]{Plakh-Arx-Sao}\label{rem:2.6}
For any $x_1, x_2\in \{0; 1\}$ and $t\in [0, 1]$ we have that
$$
\Rot^{\Rot^{x_1}(x_2)}(t) = \Rot^{x_1+x_2}(t).
$$
\end{remark}

The next follows directly from Remark~\ref{rem:2.6}.

\begin{remark}\label{rem:2.7}
For every $x\in \{0; 1\}$ and $t\in [0, 1]$ we have
$$\Rot^{\Rot^x(0)}(t) = \Rot^x(t).$$
\end{remark}

\begin{remark}\label{rem:2.8}
It follows from Notations~\ref{not:2.1} and~\ref{not:2.4} that

(i) $ \ikspm{n} = \iks{n} +\delta_{n+1,k_n}\cdot \lI_{n+1,k_n}$,

(ii) $\lI_{n+2,k_{n+1}} = \Rot^{x_{n+1}}(\delta_{n+1,k_n})\cdot
\lI_{n+1,k_n}$,

(iii) $\delta_{n+1,k_n} \cdot \lI_{n+1,k_n} = \lI_{n+2,k_{n+1}}
+\dGt_{n+1}\cdot (1-2\cdot \lI_{n+2,k_{n+1}})$.
\end{remark}

\begin{lemma}\cite[Lema~2.11]{Plakh-Arx-Sao}\label{lema:2.9}
For any $n\geq 1$ we have that:

(i) $\iks{n+1} =\iks{n} +\dGt_{n+1}\cdot \delta_{n+1,k_n}\cdot
\lI_{n+1,k_n}$,

(ii) $\iksp{n+1} =\iksp{n} - (1-\dGt_{n+1})\cdot
(1-\delta_{n+1,k_n})\cdot \lI_{n+1,k_n}$.
\end{lemma}

\begin{lemma}\label{lema:2.10}
For any $n\geq 1$ we have that:

$$\ikspm{n+1} = \ikspm{n} +(-1)^{1+\dGt_{n+1}}\cdot
\Rot^{1+\dGt_{n+1}}(\delta_{n+2,k_{n+1}})\cdot
\lI_{n+2,k_{n+1}}.$$
\end{lemma}

\begin{proof}
By (i) of Remark~\ref{rem:2.8},
\begin{equation}\label{eq:2.1}
\iks{n} = \ikspm{n} -\delta_{n+1,k_n}\cdot \lI_{n+1,k_n}.
\end{equation}
Also, plug $n+1$ into~(i) of ~\ref{rem:2.8}, whence
$$
\ikspm{n+1} = \iks{n+1} +\delta_{n+2,k_{n+1}}\cdot
\lI_{n+2,k_{n+1}} \stackrel{\text{(i) of Lem.~\ref{lema:2.9}}}{=}
$$$$
=\iks{n} + \dGt_{n+1}\cdot \delta_{n+1,k_n}\cdot \lI_{n+1,k_n}
+\delta_{n+2,k_{n+1}}\cdot
\lI_{n+2,k_{n+1}}\stackrel{\text{by~\eqref{eq:2.1}}}{=}
$$$$
=\ikspm{n} -\delta_{n+1,k_n}\cdot \lI_{n+1,k_n} +\dGt_{n+1}\cdot
\delta_{n+1,k_n}\cdot \lI_{n+1,k_n} +\delta_{n+2,k_{n+1}}\cdot
\lI_{n+2,k_{n+1}}=
$$$$
=\ikspm{n} +(\dGt_{n+1}-1)\cdot \delta_{n+1,k_n} \cdot
\lI_{n+1,k_n} +\delta_{n+2,k_{n+1}}\cdot
\lI_{n+2,k_{n+1}}\stackrel{\text{by~(iii) of
Rem.~\ref{rem:2.8}}}{=}
$$$$
=\ikspm{n} +(\dGt_{n+1}-1)\cdot (\lI_{n+2,k_{n+1}} +
\dGt_{n+1}\cdot (1-2\cdot \lI_{n+2,k_{n+1}}))
+\delta_{n+2,k_{n+1}}\cdot \lI_{n+2,k_{n+1}}.
$$
If $\dGt_{n+1}=0$, then $$ \ikspm{n+1} = \ikspm{n} -
\lI_{n+2,k_{n+1}} +\delta_{n+2,k_{n+1}}\cdot \lI_{n+2,k_{n+1}}
=\ikspm{n} -\Rot(\delta_{n+2,k_{n+1}})\cdot \lI_{n+2,k_{n+1}}.
$$ If $\dGt_{n+1}=1$, then $$ \ikspm{n+1} = \ikspm{n}
+\delta_{n+2,k_{n+1}}\cdot \lI_{n+2,k_{n+1}}.
$$
In each of these the lemma follows.
\end{proof}

\begin{remark}\cite[Rem.~2.17]{Plakh-Arx-Sao}\label{rem:2.11}
For any $n\geq 1$ and all $i, 1\leq i\leq n$, one have $
g^i(I_{n+1,k_n}) = I_{n+1-i,\Rot_{n-i}^{\dGt_i} (k_{i+1, n})}. $
\end{remark}

We have obtained the next fact duffing the proof
of~\cite[Remark~2.19]{Plakh-Arx-Sao}.

\begin{remark}\label{rem:2.12}
For every $n\geq 1$ and for all $i\leq n$ we have
$$ \delta_{n+1,k} =
\Rot^{\dGt_{i}}( \delta_{n+1-i,\Rot_{n-i+1}^{\dGt_{i}}(k_{i,n})}).
$$
\end{remark}

\section{Invariant coordinates, $g$-sequences and
$g$-decompositions}\label{relat}

Notice, that there is strong relation between the $g$-expansion
$(\dGt_i)_{i\geq 0}$ and digit sequence.

\begin{lemma}\label{lema:3.1}
Let $\varrho_i$ be $g$-digit sequence of $x$. Then

(i) $\varrho_i =\Rot^{\dGt_{i-1}}(\dGt_i)$, and

(ii) $\dGt_i = \Rot^{\sum\limits_{j=1}^i\ChEps_j}(0)$ for $i\geq
1$.
\end{lemma}

\begin{proof}
Part (i) follow from Remark~\ref{rem:2.11}.

By definitions write
$$
\dGt_1 = \ChEps_1.
$$
Now, by~(i), obtain
$$
\dGt_2 =\Rot^{\ChEps_2}(\ChEps_1) = \Rot^{\ChEps_2 +\ChEps_1}(0).
$$ Thus, part~(ii) follows from~(i) by induction.
\end{proof}

\begin{lemma}
There is the correspondence between $g$-expansions and invariant
coordinates, precisely:

1. The rule of the construction of $\theta_{i-1}$ by $\dGt_i$ is
as below:

a. of  If $\dGt_i =0$, but not all $\{\dGt_k, k>i\}$ are zero,
then $\theta_{i-1} = -1$;

b. If $\dGt_k=0$ for all $k\geq i$, then $\theta_{i-1} = 0$.

c. If $\dGt_i = 1$, then $\theta_{i-1} = 1$.

2. The rule of the construction of $\dGt_{i+1}$ by $\theta_i$ is
as below: if $\theta_i < 1$, then $\dGt_{i+1} =0$, otherwise
$\dGt_{i+1} =1$.
\end{lemma}

\section{Skew tent maps}\label{skew}

We will specify in this section the results about $g$-expansion of
carcass maps to the case of skew carcass map. Till the end of this
section let fix arbitrary number $v\in (0, 1)$, skew tent map
$f_v$, and $x\in [0, 1]$ with $f_v$-decomposition $(\dGt_i)_{i\geq
0}$. We will show in this section that some our results
of~\cite{Chaos}, which we have obtained about skew tent maps, are
partial case of more general reasonings about unimodal maps.

\begin{lemma}\cite[Lema~8]{Chaos}\label{lema:4.1}
For every $n\geq 1$ we have that
$$\delta_{k_n} =\Rot^{\dGt_n}(v).$$
\end{lemma}

\begin{proof}
Plug $i=n$ into Remark~\ref{rem:2.12} and obtain $$
\delta(I_{n+1,k}) = \Rot^{\dGt_{n}}(
\delta(I_{1,\Rot_{1}^{\dGt_{n}}(\dGt_n)})) = \Rot^{\dGt_n}(v)$$
and we are done.
\end{proof}

\begin{remark}\label{rem:4.2}
Notice that Lemma~\ref{lema:4.1} (i.e. Lemma~8 in~\cite{Chaos})
was formulated in~\cite{Chaos} as follows: for each $n\geq 1$ the
equality $\ikspm{n} -\iks{n} = (\iksp{n} -\iks{n})\cdot
\myae(\dGt_n, 0)$ holds, where $$ \myae(a,\, b) = \left\{
\begin{array}{ll} v & \text{if
}a=b,\\
1-v& \text{if }a\neq b\end{array}\right. $$ for all $a, b\in
\{0,\, 1\}$.
\end{remark}

\begin{remark}\cite[Lemma~11]{Chaos}
\label{rem:4.3} For every $n\geq 1$ we have
$$ \lI_{n+1,k_n} = \lI_{n,k_{n-1}}\cdot\Rot^{\dGt_{n-1} +\dGt_n}(v).
$$
\end{remark}

\begin{proof}
By (ii) of Remark~\ref{rem:2.8} write
$$ \lI_{n+1,k_n} = %
\lI_{n,k_{n-1}}\cdot \Rot^{\dGt_n}(\delta_{k_{n-1}}) %
\stackrel{\text{Lem.~\ref{lema:4.1}}}{=}
\lI_{n,k_{n-1}}\cdot\Rot^{\dGt_{n-1} +\dGt_n}(v).
$$\end{proof}

\begin{remark}
Notice that Lemma~\ref{rem:4.3} (i.e. Lemma~11 in~\cite{Chaos})
was formulated in~\cite{Chaos} as follows: for every $n\geq 1$ we
have $\iksp{n+1} -\iks{n+1} = (\iksp{n} -\iks{n})\cdot
\myae(\dGt_n, \dGt_{n+1})$, where $\myae$ means the same as in
Remark~\ref{rem:4.2}.
\end{remark}

We will need the next fact for our further computations.

\begin{lemma}\label{rem:4.5}
For every $n\geq 1$ we have $$ \lI_{n+1,k_n} =
\prod\limits_{i=0}^{n-1} \Rot^{\dGt_i +\dGt_{i+1}}(v).
$$
\end{lemma}

\begin{proof}
It follows from Lemma~\ref{rem:4.3} by induction on $n$  that
$\lI_{n+1,k_n} = \prod\limits_{i=1}^{n} \Rot^{\dGt_{n+1-i}
+\dGt_{n-i}}(v).$ Now, change $i$ to $n-i$, and we are done.
\end{proof}

We are ready now to prove Theorem~\ref{th:4}.

\begin{proof}[Proof of Theorem~\ref{th:4}]
Keeping in mind Lemma~\ref{rem:4.5}, notice that $\Rot^{\dGt_i
+\dGt_{i+1}}(v)$ in the product $\prod\limits_{i=0}^{n-1}
\Rot^{\dGt_i +\dGt_{i+1}}(v)$ equals either $v$, or $1-v$ and,
moreover the exponent of $(1-v)$ in the product is
$\sum\limits_{i=0}^{n-1} |\dGt_{i+1} -\dGt_i|$. Now, simplify
$$
\sum\limits_{i=0}^{n-1}  |\dGt_{i+1} -\dGt_i| =
\sum\limits_{i=0}^{n-1} \Rot^{\dGt_i}(\dGt_{i+1})
\stackrel{\text{Lem.~\ref{lema:3.1}}}{=} \sum\limits_{i=0}^{n-1}
\varrho_{i+1},
$$ and we are done by Lemma~\ref{rem:4.5}.
\end{proof}



\section{Explicit formulas for the conjugacy}\label{explicit}

\begin{lemma}\label{lema:5.1}
Let $(\dGt_i)_{i\geq 0}$ be the $g$-decomposition of a number
$x\in [0, 1]$. Then for every $n\geq 1$ we have that:

(i) $\iks{n} =\sum\limits_{i=1}^n \dGt_i\cdot
\delta_{k_{i-1}}\cdot \lI_{i,k_{i-1}}$

(ii) $ \iksp{n} = 1 -\sum\limits_{i=1}^n (1-\dGt_i)\cdot
(1-\delta_{k_{i-1}})\cdot \lI_{i,k_{i-1}};$

(iii) $ \ikspm{n} = v + \sum\limits_{i=1}^n (-1)^{1+\dGt_i}\cdot
\Rot^{1+\dGt_i}(\delta_{k_i})\cdot \lI_{i+1,k_i}.$
\end{lemma}

\begin{proof}
Notice that $\iks{0} = 0$; $\ikspm{0} =v$ and $\iksp{0} =1$. Then,
the lemma follows from Lemmas~\ref{lema:2.9} and~\ref{lema:2.10}.
\end{proof}

\begin{lemma}
For every $n\geq 1$ we have
$$ \iks{n} =
\dGt_1\cdot v +\sum\limits_{n=2}^\infty \dGt_n\cdot
\Rot^{\dGt_{n-1}}(v) \cdot \prod\limits_{i=0}^{n-2} \Rot^{\dGt_i
+\dGt_{i+1}}(v).
$$
\end{lemma}

\begin{proof} By~(i) of Lemma~\ref{lema:5.1},
$$ \iks{n} =
\sum\limits_{i=1}^n \dGt_i\cdot \delta_{k_{i-1}}\cdot
\lI_{i,k_{i-1}}(x) \stackrel{\text{Lem.~\ref{lema:4.1}}}{=}
\dGt_1\cdot v +\sum\limits_{i=2}^n \dGt_i\cdot
\Rot^{\dGt_{i-1}}(v) \cdot \lI_{i,k_{i-1}}(x),
$$
and the lemma follows from Lemma~\ref{rem:4.5}.
\end{proof}

\begin{lemma}
For any $v\in (0, 1)$, any $x\in (0, 1)$ with $f_v$-digit sequence
$(\ChEps_i)_{i\geq 1}$ the equality $$ \iks{n} = \ChEps_1\cdot v
+\sum\limits_{i=2}^n \Rot^{\sum\limits_{j=1}^i \ChEps_j}(0) \cdot
\Rot^{\sum\limits_{j=1}^{i-1}\ChEps_j}(v) \cdot
(1-v)^{\sum\limits_{j=1}^{i-1} \varrho_j}\cdot
v^{i-1-\sum\limits_{j=1}^{i-1} \varrho_j}
$$ holds.
\end{lemma}

\begin{proof}

By~(i) of Lemma~\ref{lema:5.1}, $$ \iks{n} = \sum\limits_{i=1}^n
\dGt_i\cdot \delta_{k_{i-1}}\cdot \lI_{i,k_{i-1}}(x)
\stackrel{\text{Lem.~\ref{lema:4.1}}}{=}
\dGt_1\cdot v +\sum\limits_{i=2}^n \dGt_i\cdot
\Rot^{\dGt_{i-1}}(v) \cdot \lI_{i,k_{i-1}}(x)
\stackrel{\text{Lem.~\ref{lema:3.1}}}{=}
$$$$
= \varrho_1\cdot v +\sum\limits_{i=2}^n
\Rot^{\sum\limits_{j=1}^i\ChEps_j}(0)\cdot
\Rot^{\Rot^{\sum\limits_{j=1}^{i-1}\ChEps_j}(0)}(v) \cdot
\lI_{i,k_{i-1}}(x) \stackrel{\text{Rem.~\ref{rem:2.7}}}{=}
$$$$= \varrho_1\cdot v +\sum\limits_{i=2}^n
\Rot^{\sum\limits_{j=1}^i\ChEps_j}(0)\cdot
\Rot^{\sum\limits_{j=1}^{i-1}\ChEps_j}(v) \cdot
\lI_{i,k_{i-1}}(x),
$$ and we are done by Theorem~\ref{th:4}.
\end{proof}

\begin{lemma}
Let $x\in [0, 1]$ has $f_v$-digits $(\varrho_i)_{i\geq 1}$. Then
$$\iksp{n} =1 -(1-\dGt_1)\cdot (1-v) -\sum\limits_{i=2}^n (1-\dGt_i)\cdot
\Rot^{1+\dGt_{i-1}}(v)\cdot \prod\limits_{j=0}^{i-2} \Rot^{\dGt_j
+\dGt_{j+1}}(v)$$
\end{lemma}

\begin{proof}
By~(ii) of Lemma~\ref{lema:5.1}, $$ \iksp{n} = 1
-\sum\limits_{i=1}^n (1-\dGt_i)\cdot (1-\delta_{k_{i-1}})\cdot
\lI_{i,k_{i-1}} \stackrel{\text{Lem.~\ref{lema:4.1}}}{=}
$$$$
= 1 -\sum\limits_{i=1}^n (1-\dGt_i)\cdot
(1-\Rot^{\dGt_{i-1}}(v))\cdot \lI_{i,k_{i-1}}
\stackrel{\text{Lem.~\ref{rem:4.5}}}{=}
$$$$
= 1 -(1-\dGt_1)\cdot (1-v) -\sum\limits_{i=2}^n (1-\dGt_i)\cdot
\Rot^{1+\dGt_{i-1}}(v)\cdot \prod\limits_{j=0}^{i-2} \Rot^{\dGt_j
+\dGt_{j+1}}(v)
$$ and we are done.
\end{proof}

\begin{lemma}
Let $x\in [0, 1]$ has $f_v$-digits $(\varrho_i)_{i\geq 1}$. Then
$$\iksp{n} =1 -(1-\ChEps_1)\cdot (1-v) -\sum\limits_{i=2}^n
\Rot^{1+\sum\limits_{j=1}^i\ChEps_j}(0)\cdot
\Rot^{1+\sum\limits_{j=1}^{i-1}\ChEps_j}(v)\cdot
(1-v)^{\sum\limits_{i=1}^{n-1} \ChEps_i}\cdot
v^{n-1-\sum\limits_{i=1}^{n-1} \ChEps_i}.$$
\end{lemma}

\begin{proof}
By~(ii) of Lemma~\ref{lema:5.1}, $$ \iksp{n} = 1
-\sum\limits_{i=1}^n (1-\dGt_i)\cdot (1-\delta_{k_{i-1}})\cdot
\lI_{i,k_{i-1}} \stackrel{\text{Lem.~\ref{lema:4.1}}}{=}
$$$$
= 1 -\sum\limits_{i=1}^n (1-\dGt_i)\cdot
(1-\Rot^{\dGt_{i-1}}(v))\cdot \lI_{i,k_{i-1}}
\stackrel{\text{Th.~\ref{th:4}}}{=}
$$$$= 1 -(1-\dGt_1)\cdot (1-v) -\sum\limits_{i=2}^n (1-\dGt_i)\cdot
\Rot^{1+\dGt_{i-1}}(v)\cdot (1-v)^{\sum\limits_{i=1}^{n-1}
\varrho_i}\cdot v^{n-1-\sum\limits_{i=1}^{n-1} \varrho_i}
\stackrel{\text{Lem.~\ref{lema:3.1}}}{=}
$$$$
=1 -(1-\ChEps_1)\cdot (1-v) -\sum\limits_{i=2}^n
\Rot^{1+\sum\limits_{j=1}^i\ChEps_j}(0)\cdot
\Rot^{1+\Rot^{\sum\limits_{j=1}^{i-1}\ChEps_j}(0)}(v)\cdot
(1-v)^{\sum\limits_{i=1}^{n-1} \ChEps_i}\cdot
v^{n-1-\sum\limits_{i=1}^{n-1}
\ChEps_i}\stackrel{\text{Rem.~\ref{rem:2.6}}}{=}
$$$$
=1 -(1-\ChEps_1)\cdot (1-v) -\sum\limits_{i=2}^n
\Rot^{1+\sum\limits_{j=1}^i\ChEps_j}(0)\cdot
\Rot^{1+\sum\limits_{j=1}^{i-1}\ChEps_j}(v)\cdot
(1-v)^{\sum\limits_{i=1}^{n-1} \ChEps_i}\cdot
v^{n-1-\sum\limits_{i=1}^{n-1} \ChEps_i},
$$ which is necessary.
\end{proof}

\begin{lemma}
Let $x\in [0, 1]$ has $f_v$-digits $(\varrho_i)_{i\geq 1}$. Then
$$
\ikspm{n} = v + \sum\limits_{i=1}^{n} (-1)^{\dGt_i+1}\cdot
\Rot^{\dGt_i+1 +\dGt_i}(v)\cdot \prod\limits_{i=0}^{n-1}
\Rot^{\dGt_i +\dGt_{i+1}}(v)
$$
\end{lemma}

\begin{proof}
By~(iii) of Lemma~\ref{lema:5.1}, $$ \ikspm{n} = v+
\sum\limits_{i=1}^{n} (-1)^{\dGt_i+1}\cdot
\Rot^{\dGt_i+1}(\delta_{i+1,k_i})\cdot \lI_{i+1}
\stackrel{\text{by Lem.~\ref{lema:4.1}}}{=}$$$$ =v +
\sum\limits_{i=1}^{n} (-1)^{\dGt_i+1}\cdot \Rot^{\dGt_i+1
+\dGt_i}(v)\cdot \lI_{i+1},
$$
and the lemma follows from Lemma~\ref{rem:4.5}.
\end{proof}

\begin{lemma}
Let $x\in [0, 1]$ has $f_v$-digits $(\varrho_i)_{i\geq 1}$. Then
$$
\ikspm{n} =v + \sum\limits_{i=1}^{n}
(v-1)^{1+\sum\limits_{j=1}^{i} \varrho_j}\cdot
v^{i-\sum\limits_{j=1}^{i} \varrho_j}
$$
\end{lemma}

\begin{proof}
By~(iii) of Lemma~\ref{lema:5.1}, $$ \ikspm{n} = v +
\sum\limits_{i=1}^{n} (-1)^{\dGt_i+1}\cdot
\Rot^{\dGt_i+1}(\delta_{i+1,k_i})\cdot \lI_{i+1}
\stackrel{\text{by Lem.~\ref{lema:4.1}}}{=}$$$$ =v +
\sum\limits_{i=1}^{n} (-1)^{\dGt_i+1}\cdot \Rot^{\dGt_i+1
+\dGt_i}(v)\cdot \lI_{i+1} \stackrel{\text{by Th.~\ref{th:4}}}{=}
$$$$=v + \sum\limits_{i=1}^{n} (-1)^{\dGt_i+1}\cdot (1-v)\cdot
(1-v)^{\sum\limits_{j=1}^{i} \varrho_j}\cdot
v^{i-\sum\limits_{j=1}^{i} \varrho_j} \stackrel{\text{by
Lem.~\ref{lema:3.1}}}{=}
$$$$
=v + \sum\limits_{i=1}^{n}
(-1)^{1+\Rot^{\sum\limits_{j=1}^i\ChEps_j}(0)}\cdot
(1-v)^{1+\sum\limits_{j=1}^{i} \varrho_j}\cdot
v^{i-\sum\limits_{j=1}^{i} \varrho_j}=
$$$$
=v + \sum\limits_{i=1}^{n} (v-1)^{1+\sum\limits_{j=1}^{i}
\varrho_j}\cdot v^{i-\sum\limits_{j=1}^{i} \varrho_j}
$$ and we are done.
\end{proof}

\begin{proof}[Proof of Theorem~\ref{th:3}]
Denote $a_n = \iks{n} + \dGt_{n+1}\cdot \lI_{n+1,k_n}$.

$$a_{n+1} = \iks{n+1} + \dGt_{n+2}\cdot \lI_{n+2,k_{n+1}}
\stackrel{\text{(i) of Lem.~\ref{lema:2.9}}}{=}$$
$$=\iks{n} +
\dGt_{n+1}\cdot \delta_{n+1,k_n}\cdot \lI_{n+1,k_n} +
\dGt_{n+2}\cdot \lI_{n+2,k_{n+1}} %
\stackrel{\text{(ii) of Rem.~\ref{rem:2.8}}}{=} $$ $$%
=\iks{n} + \dGt_{n+1}\cdot \delta_{n+1,k_n}\cdot \lI_{n+1,k_n} +
\dGt_{n+2}\cdot \Rot^{\dGt_{n+1}}(\delta_{n+1,k_n})\cdot
\lI_{n+1,k_n}=
$$$$
a_n + (\dGt_{n+1}\cdot \delta_{n+1,k_n} -\dGt_{n+1}
+\dGt_{n+2}\cdot \Rot^{\dGt_{n+1}}(\delta_{n+1,k_n}))\cdot
\lI_{n+1,k_n}
$$

If $\dGt_{n+1}=0$, then $$ a_{n+1} = a_n + \dGt_{n+2}\cdot
\delta_{n+1,k_n}\cdot \lI_{n+1,k_n}
$$

If $\dGt_{n+1} =1$, then $$ a_{n+1} = a_n + ( \delta_{n+1,k_n} -1
+\dGt_{n+2}\cdot \Rot(\delta_{n+1,k_n}))\cdot \lI_{n+1,k_n}=
$$$$= a_n + ( \dGt_{n+2}\cdot \Rot(\delta_{n+1,k_n})
-\Rot(\delta_{n+1,k_n}))\cdot \lI_{n+1,k_n}=
$$$$
=a_n -\Rot(\dGt_{n+2})\cdot \Rot(\delta_{n+1,k_n})\cdot
\lI_{n+1,k_n}
$$

In general,
$$
a_{n+1} = a_n + (-1)^{\dGt_{n+1}}\cdot
\Rot^{\dGt_{n+1}}(\dGt_{n+2}) \cdot
\Rot^{\dGt_{n+1}}(\delta_{n+1,k_n})\cdot \lI_{n+1,k_n}
\stackrel{\text{(ii) of Rem.~\ref{rem:2.8}}}{=}
$$$$=a_n + (-1)^{\dGt_{n+1}}\cdot \Rot^{\dGt_{n+1}}(\dGt_{n+2}) \cdot
\lI_{n+2,k_{n+1}}
$$

Since $a_0 =\dGt_1$, then for any $n\geq 1$ we have
$$
a_n = \dGt_1 +\sum\limits_{i=1}^{n} (-1)^{\dGt_i}\cdot
\Rot^{\dGt_i}(\dGt_{i+1}) \cdot \lI_{i+1,k_i}
\stackrel{\text{Lem.~\ref{lema:3.1}}}{=}
$$$$
=\ChEps_1 +\sum\limits_{i=1}^{n}
(-1)^{\sum\limits_{j=1}^i\ChEps_j}\cdot \ChEps_{i+1} \cdot
\lI_{i+1,k_i} = \ChEps_1 +\sum\limits_{i=2}^{n+1}
(-1)^{\sum\limits_{j=1}^{i-1}\ChEps_j}\cdot \ChEps_i \cdot
\lI_{i,k_{i-1}}.
$$
\end{proof}

\setlength{\unitlength}{1pt}

\pagestyle{empty}
\bibliography{Ds-Bib}{}
\bibliographystyle{makar}


\end{document}